\documentclass[preprint]{elsarticle}

\makeatletter
\def\ps@pprintTitle{%
	\let\@oddhead\@empty
	\let\@evenhead\@empty
	\def\@oddfoot{\footnotesize\itshape
		{} \hfill\today}%
	\let\@evenfoot\@oddfoot
}
\makeatother

\usepackage{latexsym}
\usepackage{lineno}
\usepackage{hyperref}
\usepackage{indentfirst}
\usepackage{amsxtra}
\usepackage{amssymb}
\usepackage{amsthm}
\usepackage{natbib}
\usepackage{amsmath}
\usepackage{tikz}
\usepackage{amscd}
\usepackage{MnSymbol}

\usepackage[capitalise]{cleveref}
\newtheorem{theor}{Theorem}
\newtheorem{prop}[theor]{Proposition}
\newtheorem{cor}[theor]{Corollary}
\newtheorem{lemma}[theor]{Lemma}
\newtheorem{question}[theor]{Question}

\theoremstyle{definition} 
\newtheorem{defin}{Definition}
\newtheorem{rem}{Remark}

\newtheorem*{conv}{Convention}
\newtheorem{ex}[theor]{Example}

\DeclareMathOperator{\Aut}{Aut}

\DeclareMathOperator{\SmallBrace}{SmallBrace}
\DeclareMathOperator{\SmallSkewbrace}{SmallSkewbrace}

\begin{document}

\begin{frontmatter}
	\title{A note on semiprime skew left braces and related semidirect products
 }
	\tnotetext[mytitlenote]{The author is a member of GNSAGA (INdAM). The author was partially supported by the MAD project Cod. ARS$01\_00717$}
	\author{Marco CASTELLI}
	\ead{marco.castelli@unisalento.it - marcolmc88@gmail.com}
 
	\address[unile]{Dipartimento di Matematica e Fisica ``Ennio De Giorgi"
		\\
		Universit\`{a} del Salento\\
		Via Provinciale Lecce-Arnesano \\
	73100 Lecce (Italy)\\}

\begin{abstract}
In this paper, we focus on semiprime skew left braces provided by semidirect products. We show that if a semidirect product $B_1\rtimes B_2$ is semiprime and $B_1$ is Artinian, then $B_1$ must be semiprime. Moreover, we prove that the semidirect product of strongly semiprime skew left braces is strongly semiprime. 
Finally, addressing Question $1$ in Smoktunowicz, \emph{More on skew braces and their ideals} (2024), we provide examples of skew left braces of abelian type that are non-simple and strongly prime.
\end{abstract}
\begin{keyword}
\texttt{Yang-Baxter equation\sep skew brace \sep left brace }
\MSC[2020] 16T25\sep 81R50  
\end{keyword}

\end{frontmatter}

\section*{Introduction}

Motivated by the search of set-theoretic solutions of the Yang-Baxter equation, which are of great impact in the study of the quantum Yang-Baxter equation (see \cite{baxter1972partition,yang1967} for more details), several algebraic structures were introduced. In particular, in \cite{guarnieri2017skew} Guarnieri and Vendramin introduced the notion of \emph{skew left brace}. A triple $(B,+,\circ)$ is said to be a skew left brace if $(B,+)$ and $(B,\circ)$ are groups, and the equality 
$$a\circ (b+c)=a\circ b-a+a\circ c$$
holds for all $a,b,c\in B$. This definition extends the one given by Rump in \cite{rump2007braces} where $(B,+)$ is an abelian group. Skew left braces are powerful tools in the study of the quantum Yang-Baxter equation, since they allow to provide set-theoretic solutions belonging to the class of non-degenerate ones. On the other hand, these algebraic structures have gained increased interest in recent years because of several links with various topics - such as graph theory \cite{properzi2023common}, Hopf-Galois structures \cite{CHILDS2018270}, triply factorized groups \cite{ballester2022}, and regular subgroups of the affine group \cite{cacs1} - that have been discovered.\\
It is well-known that groups and rings provide skew left braces. In fact, every group $(G,\cdotp)$ has two natural skew left brace structures given by $(G,\cdotp,\cdotp)$ and $(G,\cdotp,\cdotp^{op})$ (where $\cdotp^{op}$ is the opposite operation); while a radical ring $(A,+,\cdotp)$ is a skew left brace if we replace the operation $\cdotp$ with the adjoint operation $\circ$ given by $a\circ b:=a+a\cdotp b+b$ for all $a,b\in A$. For this reason, to study skew left braces, many papers have introduced several notions and tools inspired by group theory and ring theory,  such as sub-skew left braces, left ideals, ideals \cite{guarnieri2017skew}, matched product \cite{bachiller2015extensions}, and nilpotency \cite{CeSmVe19}. The concept of a solvable skew left brace was introduced in \cite{BCJO19} for skew left braces of abelian type, and extended in \cite{KoSmVe18} for arbitrary skew left braces. In analogy to what happens with groups, it is known that if $B$ is a skew left brace and $I$ is one of its ideals, then $B$ is solvable if and only if $B/I$ and $I$ are solvable (see \cite[Lemma $6.4$]{KoSmVe18}, \cite[Lemma 8.1.5]{kinnear2019}, and \cite[Proposition $2.4$]{BCJO19}). In \cite{KoSmVe18,smoktunowicz2024more}, the notions of (strongly) prime and (strongly) semiprime ideals were introduced, and suitable versions of celebrated theorems of ring theory were given for skew left braces (see \cite[Sections $5-6$]{KoSmVe18}).\\
To concretely provide new skew left braces, in \cite{Ru08,SmVe18} a suitable version of skew left braces semidirect product was developed. This construction has recently appeared in several contexts. Indeed, in \cite[Lemma 4.5]{dietzel2024indecomposable} it was shown that the permutation skew left braces of some set-theoretical solutions arise as semidirect products of two skew left braces, while Corollary $21$ of \cite{CaCaSt20x} shows that a particular class of set-theoretic solutions can be recovered by semidirect products of skew left braces. In the last year, this construction was studied in the context of digroups \cite{facchini2024}.
As a further step in the investigation of prime and semiprime skew left braces, in \cite[Example 5.3]{CJO20} Cedò, Jespers and Okni\'nski constructed the first example of a prime non-simple skew left brace of abelian type using semidirect products. In \cite{kinnear2021wreath}, Kinnear showed that the class of semiprime skew left brace is closed under semidirect products. The goal of this paper is to continue the investigation of semidirect products and their links with (strongly) prime and semiprime skew left braces.\\
In the first part of the paper, following \cite[Question 3.4]{kinnear2021wreath} and \cite[Question 8.2.11]{kinnear2019}, we discuss a possible converse to Kinnear's result on the semidirect product of semiprime skew left braces. In this context, we show that if a skew left braces semidirect product $B_1\rtimes B_2$ is semiprime, then so is $B_1$, provided that $B_1$ is Artinian \cite{jespers2021radical} (this is the case if $B_1$ is finite). On the other hand, we show that an analogous result for $B_2$ does not follow (even in the finite case), therefore we cannot provide a characterization similar to that of solvable skew left braces. Moreover, we show a “strongly" version of Kinnear's result, i.e. we prove that the semidirect product of strongly semiprime skew left braces is strongly semiprime. In Section $4$, we focus on strongly prime skew left braces. As noted in \cite{smoktunowicz2024more}, finding finite examples of non-simple strongly prime skew left braces of abelian type seems to be quite difficult (see \cite[Question $1$]{smoktunowicz2024more}). In this regard, we will note that the method provided in \cite[Proposition $5.1$]{CJO20} allows to construct strongly prime skew left braces of abelian type: in particular, in \cite[Example $5.3$]{CJO20} a concrete remarkable example was provided without computer calculations. However, this example has $92.160$ elements. As a personal contribution in this context, we provide a family of semidirect products of strongly prime non-simple skew left braces (not necessarily of abelian type) that includes the one constructed in \cite[Proposition $5.1$]{CJO20}. Moreover, we construct a smaller “computer-free" example with $576$ elements that answers in negative sense to \cite[Question $1$]{smoktunowicz2024more}. To this aim, the classification of skew left braces of abelian type with multiplicative group isomorphic to $\mathbb{S}_4$, given in \cite{rump2019construction}, is of crucial importance. Finally, we will show that the construction of further concrete examples by our results is possible, provided that we help us by the GAP package \cite{Ve24pack} to find skew left braces that satisfy suitable conditions.

\section{Basic definitions, notations and results}

In this section, we give basic definitions and results involving skew left braces. We start with the definition of a skew left brace.

\begin{defin}[Definition $1.1$, \cite{guarnieri2017skew}]
A triple $(B,+,\circ)$ with two binary operations is said to be a \emph{skew left brace} if $(B,+)$ and $(B,\circ)$ are groups and $a\circ (b+c)=a\circ b-a+a\circ c$ for all $a,b,c\in B$. $(B,+)$ will be called the \emph{additive group}, and $(B,\circ)$ will be called the \emph{multiplicative group}. Moreover, the inverse of an element $a\in B$ in the additive (resp. multiplicative) group will be indicated by $-a$ (resp. $a^-$).
\end{defin}

\noindent Skew left braces \emph{of abelian type}, i.e. with an abelian additive group, were previously introduced by Rump in \cite{rump2007braces}. 

\begin{ex}\label{primies}
\begin{itemize}
    \item[1)] Let $(B,+)$ be a group and $\circ$ be the binary operation on $B$ given by $a\circ b:= a+b$ for all $a,b\in B$. Then, the triple $(B,+,\circ)$ is a skew left brace. In the following, we will refer to this skew left brace as the \emph{trivial} skew left brace on $(B,+)$.
    \item[2)] Let $B:=(\mathbb{Z}/p^2\mathbb{Z},+)$ and $\circ$ be the binary operation on $B$ given by $a\circ b:=a+b+p\cdotp a\cdotp b$ (where $\cdotp$ is the ring-multiplication of $\mathbb{Z}/p^2\mathbb{Z}$). Then, $(B,+,\circ)$ is a skew left brace.
\end{itemize}
\end{ex}

\noindent Clearly, there are several skew left braces different from the ones provided in \cref{primies}. The skew left braces up to size $168$ (with some exceptions) were calculated in \cite{Ve24pack}. The database $\SmallBrace$ collects all the skew left braces of abelian type, while the database $\SmallSkewbrace$ collects all the skew left braces. Since throughout the paper we will ofter refer to these databases, we adopt the following convention.

\begin{conv}
    We will indicate by $\mathcal{B}_{n,k}$ the skew left brace of abelian type having size $n$ and located at position $k$ in the database $\SmallBrace$. Moreover, we will indicate by $\mathcal{SB}_{n,k}$ the skew left brace having size $n$ and located at position $k$ in the database $\SmallSkewbrace$.
\end{conv}

\noindent Given a skew left brace $B$ and $a\in B$, let us denote by $\lambda_a:B\longrightarrow B$ the map from $B$ into itself defined by $\lambda_a(b):= - a + a\circ b,$ for all $b\in B$. Then, $\lambda_a\in\Aut(B,+)$, for every $a\in B$; and the map $\lambda$ from $B$ to $Aut(B,+)$ given by $\lambda(a):=\lambda_a$ is an action of $(B,\circ)$ on $(B,+)$ by automorphisms (see \cite{guarnieri2017skew} for further details). As for other algebraic structures, one can define several sub-structures, such as left ideals and ideals.
\begin{defin}
Let $B$ be a skew left brace. A subset $I$ of $B$ is said to be a \textit{left ideal} if it is a subgroup of the additive group and $\lambda_a(I)\subseteq I$, for every $a\in B$. Moreover, an \emph{ideal} is a left ideal which is also a normal subgroup of the additive and multiplicative groups. 
\end{defin}

\noindent Of course, the subsets $\{0\}$ and $B$ are ideals, which we will call \emph{trivial.} If there are no other ideals, $B$ will be called \emph{simple}. In a standard manner, ideals give rise to \emph{quotient skew left braces}. Ideals of a quotient skew left brace $B/I$ correspond bijectively to ideals of $B$ that contain $I$. As one can expect, the intersection and the sum of a family (not necessarily finite) of ideals are again ideals. In particular, if $\mathcal{I}:=\{I_j\}_{j\in J}$ is a family of ideals, we denote the sum of the ideals in $\mathcal{I}$ by $\Sigma_{j\in J}I_j$ and by $I_1+...+I_n$ if $|J|<+\infty$. A map $\alpha$ between two skew left braces $B_1$ and $B_2$ is a \emph{homomorphism} if $\alpha(a+b)=\alpha(a)+\alpha(b)$ and $\alpha(a\circ b)=\alpha(a)\circ \alpha(b)$ for all $a,b\in B_1$, and its \emph{kernel}, defined as in the context of rings, is an ideal of $B_1$.\\
Examples of skew left braces can be constructed by a suitable version of the semidirect product, provided in \cite{SmVe18} (and in \cite{Ru08} for skew left braces of abelian type).
If $B_1$ and $B_2$ are skew left braces and $\alpha$ is a group homomorphism from $(B_2,\circ)$ to $Aut(B_1,+,\circ)$, the \emph{semidirect product of the skew left braces $B_1$ and $B_2$ via $\alpha$} is given by
$$(a_1,a_2)+(b_1,b_2):=(a_1+b_1,a_2+b_2) $$
$$(a_1,a_2)\circ(b_1,b_2):=(a_1\circ \alpha_{a_2}(b_1),a_2\circ b_2) $$
for all $(a_1,a_2),(b_1,b_2)\in B_1\times B_2 $. From now on, we will indicate this skew left brace by $B_1\rtimes_{\alpha}B_2$. In analogy to the groups semidirect product, the set $B_1\times \{0\}$ is an ideal of $B_1\rtimes_{\alpha} B_2$.

\smallskip

Before presenting the next result, whose proof is straightforward, we fix some notations. If $X,Y$ are nonempty sets, we denote by $\pi_1$ (resp. $\pi_2$) the map from $X\times Y$ to $X$ (resp. from $X\times Y$ to $Y$) given by $\pi_1(x,y):=x$ (resp. $\pi_2(x,y):=y$) for all $(x,y)\in X\times Y$. 

\begin{lemma}\label{projid}
    Let $B_1$ and $B_2$ be skew left braces, and $\alpha$ be a homomorphism from $B_2$ to $Aut(B_1,+,\circ)$. Let $B:=B_1\rtimes_{\alpha} B_2$. Then:
    \begin{itemize}
        \item[1)] if $I$ is an ideal of $B$, then $\pi_2(I)$ is an ideal of $B_2$
        \item[2)] if $J$ is an ideal of $B_2$ contained in $Ker(\alpha)$, then $\{0\}\times J$ is an ideal of $B$;
        \item[3)] if $J$ is an ideal of $B_1$, and $\alpha_b(J)=J$ for all $b\in B_2$, then $J\times \{0\}$ is an ideal of $B$.
    \end{itemize}
\end{lemma}

In order to study skew left braces using tools that resemble the ones used in ring-theory, one has to consider a third operation $*$ for a skew left brace $B$ given by $a*b:=-a+a\circ b-b$, for all $a,b\in B$. From now on, if $B$ is a skew left brace and $X,Y\subseteq B$, we denote by $X*Y$ the additive subgroup of $B$ generated by the elements of the form $x*y$, with $x\in X$ and $y\in Y$.\\
By the operation $*$, in \cite[Definitions $5.1$ and $5.6$]{KoSmVe18} the concepts of prime and semiprime skew left braces were defined. In \cite[Definitions $3$ and $4$]{smoktunowicz2024more} and \cite[Definition $5.1$]{trappeniers2023two} versions of strong (semi)primality were given.

\begin{defin}
    Let $B$ be a skew left brace.
            \vspace{-2.5mm}
    \begin{itemize}
        \item $B$ is \emph{semiprime} if for each non-zero ideal $I$ one has $I*I\neq \{0\}$.
        \vspace{-2.5mm}
        \item $B$ is \emph{prime} if for any non-zero ideals $I,J$ one has $I*J\neq \{0\}$.
                \vspace{-2.5mm}
        \item $B$ is \emph{strongly semiprime} if for each non-zero ideal $I$ one has that every $*$-product of copies of $I$ is non-zero.
                \vspace{-2.5mm}
        \item $B$ is \emph{strongly prime} if every $*$-product of any number of non-zero ideals is non-zero. 
    \end{itemize} 
\end{defin}

In \cite{kinnear2021wreath} Kinnear showed that semiprime skew left braces are closed under semidirect products. 

\begin{theor}[Corollary 2.8 of \cite{kinnear2021wreath}]\label{teokin}
    Let $B_1,B_2$ be semiprime skew left braces. Then, a semidirect product $B:=B_1\rtimes B_2$ is a semiprime skew left brace.
\end{theor}

\noindent In \cite{kinnear2021wreath} and \cite{kinnear2019} a possible converse of \cref{teokin} was raised.

\begin{question}[Question 3.4 of \cite{kinnear2021wreath} - Question $8.2.11$ of \cite{kinnear2019}]\label{ques}
    Let $B_1,B_2$ be skew left braces, and suppose that their semidirect product via $\alpha$ $B_1\rtimes_{\alpha}B_2$ is semiprime. Are $B_1$ and $B_2$ semiprime?
\end{question}

\noindent The focus of the next section will be the discussion of \cref{ques}.

\vspace{-3mm}

\section{Semiprime skew left braces in semidirect products}
In this section, we discuss a possible converse to \cref{teokin}. In the main result, we show that if $B_1\rtimes_{\alpha} B_2$ is a semiprime skew left brace, then so is $B_1$ whenever it is an Artinian skew left brace (in particular, this is the case if $B_1$ has finite order). Under this hypothesis, our result answers affirmatively to \cref{ques} on semiprimality of $B_1$. We finish the section showing by a suitable example that an analogous result for $B_2$ does not follow, even if we restrict to the finite case.
\smallskip

At first, we have to provide some technical results. Following the notation adopted in \cite[Chapter $1$]{robinson2012course}, if $B$ is a skew left brace, $j\in B$, and $X,Y\subseteq B$, we indicate by $X^j$ the set $\{j+x-j\mid x\in X\}$ and by $X^Y$ the set  $\{y+x-y\mid x\in X,y\in Y\}$.

\begin{lemma}\label{prodid}
    Let $B$ be a skew left brace, $n$ a natural number and $J,I_1,...,I_n$ distinct minimal ideals of $B$. Then $J*(I_1+...+I_n)=(I_1+...+I_n)*J=0$.
\end{lemma}

\begin{proof}
    At first, we show the equality $J*(I_1+...+I_n)=0$ by induction on $n$. If $n=1$, then $J*I_1\subseteq I_1\cap J=\{0\}$. Now, if $n>1$, $j\in J$, and $i_1\in I_1,...,i_n\in I_n$ then
    \begin{eqnarray}
       & & j*(i_1+...+i_n)= \lambda_{j}(i_1+...+i_n)-(i_1+...+i_n) \nonumber \\
       &=& \lambda_{j}(i_1)-i_1+i_1+\lambda_{j}(i_2+...+i_n)-(i_2+...+i_n)-i_1 \nonumber     
    \end{eqnarray}
    the last member belongs to $J*I_1+(J*(I_2+...+I_n))^{i_1}$, which is $0$ by the inductive hypotesis.
    Using induction again on $n$, we show the equality $(I_1+...+I_n)*J$. As in the first part, $I_1*J \subseteq I_1\cap J=\{0\}$. Now, if $n>1$, $j\in J$, and $i_1\in I_1,...,i_n\in I_n$ then
    \begin{eqnarray}
       & & (i_1+...+i_n)*j =\lambda_{i_1+...+i_n}(j)-j \nonumber \\
       &=& \lambda_{i_1}(\lambda_{\lambda_{i_1^-}(i_2+...+i_n)}(j)-j)+\lambda_{i_1}(j)-j \nonumber 
    \end{eqnarray}
the last member belongs to $\lambda_{i_1}((I_2+...+I_n)*J)+I_1*J$, which is equal to $0$ by the inductive hypothesis. Therefore, the statement follows.
\end{proof}

\begin{lemma}\label{lemchiav}
    Let $B$ be a skew left brace, $n$ be a natural number with $n>1$ and $I_1,...,I_n$ be distinct minimal ideals of $B$. Then, we have  
    $$(i_1+...+i_n)*(j_1+...+j_n)\in  (I_1*I_1)+(I_2*I_2)^{I_1}+...+(I_n*I_n)^{I_1+...+I_{n-1}}$$
    for all $i_1,j_1\in I_1,...i_n,j_n\in I_n$.
\end{lemma}

\begin{proof}
  We will prove the statement by induction on $n$. If $n=2$, let $i_1,j_1\in I_1$, and $i_2,j_2\in I_2.$ Then
    \begin{eqnarray}
      & &   (i_1+i_2)*(j_1+j_2)= \lambda_{(i_1+i_2)}(j_1+j_2)-(j_1+j_2) \nonumber \\
        &=& \lambda_{i_1}(\lambda_{\lambda_{i_1^-}(i_2)}(j_1)-j_1)+\lambda_{i_1}(j_1)+\lambda_{i_2+(-i_2+i_1+i_2)}(j_2)-j_2-j_1 \nonumber \\
        &=&  \lambda_{i_1}(\lambda_{\lambda_{i_1^-}(i_2)}(j_1)-j_1)+\lambda_{i_1}(j_1)+\lambda_{i_2}(\lambda_{\lambda_{i_2^-}(-i_2+i_1+i_2)}(j_2)-j_2)+\lambda_{i_2}(j_2)-j_2-j_1 \nonumber        
    \end{eqnarray}
and since $\lambda_{\lambda_{i_1^-}(i_2)}(j_1)-j_1\in I_2*I_1 $ and $\lambda_{\lambda_{i_2^-}(-i_2+i_1+i_2)}(j_2)-j_2\in I_1*I_2$, by \cref{prodid} we obtain that the last member is equal to $\lambda_{i_1}(j_1)-j_1+j_1+\lambda_{i_2}(j_2)-j_2-j_1 $, and hence it belongs to $(I_1*I_1)+(I_2*I_2)^{I_1}$.\\
Now, suppose $n>2$ and let $i_1,j_1\in I_1,...,i_n,j_n\in I_n$. Then, if we set $o_1:=(-(i_2+...+i_n)+i_1+(i_2+...+i_n))$, we obtain
\begin{eqnarray}
    & &(i_1+...+i_n)*(j_1+...+j_n)=\lambda_{i_1+...+i_n}(j_1+...+j_n)-(j_1+...+j_n) \nonumber \\
    &=& \lambda_{i_1}(\lambda_{\lambda_{i_1^-}(i_2+...+i_n)}(j_1)-j_1)+\lambda_{i_1}(j_1)+\lambda_{i_2+...+i_n}(\lambda_{\lambda_{(i_2+...+i_n)^-}(o_1)}(j_2+...+j_n)+ \nonumber \\
   & &  -(j_2+...+j_n))+\lambda_{i_2+...+i_n}(j_2+...+j_n)-(j_2+...+j_n)-j_1\nonumber   \\
   &=&  \lambda_{i_1}(\lambda_{\lambda_{i_1^-}(i_2+...+i_n)}(j_1)-j_1)+\lambda_{i_1}(j_1)-j_1+j_1+ \nonumber  \\
   & & +\lambda_{i_2+...+i_n}(\lambda_{\lambda_{(i_2+...+i_n)^-}(o_1)}(j_2+...+j_n)-(j_2+...+j_n))+\nonumber \\
   & & +\lambda_{i_2+...+i_n}(j_2+...+j_n)-(j_2+...+j_n)-j_1\nonumber
\end{eqnarray}
where the last member belongs to $\lambda_{i_1}((I_2+...+I_n)*I_1)+I_1*I_1+(\lambda_{i_2+...+i_n}(I_1*(I_2+...+I_n))+(I_2+...+I_n)*(I_2+...+I_n))^{j_1}$, which by \cref{prodid} and the inductive hypothesis belongs to $I_1*I_1+((I_2*I_2)+...+(I_n*I_n)^{I_2+...+I_{n-1}})^{j_1}$ which is included in $(I_1*I_1)+(I_2*I_2)^{I_1}+...+(I_n*I_n)^{I_1+...+I_{n-1}}$.
\end{proof}

\begin{lemma}\label{lempre}
    Let $(B,+,\circ)$ be a skew left brace, $I$ be a minimal ideal and $\alpha\in Aut(B,+,\circ)$. Then, $\alpha(I)$ is a minimal ideal of $B$, and furthermore $I*I=0$ if and only if $\alpha(I)*\alpha(I)=0$.
\end{lemma}

\begin{proof}
    The statement follows by a standard calculation.
\end{proof}

\begin{lemma}\label{lemind2}
    Let $(B,+,\circ)$ be a skew left brace $I_1,...,I_n$ be minimal ideals of $B$ and $I:=I_1+...+I_n$. Then, $I_i*I_i=0$ for all $i\in \{1,...,n\}$ if and only if $I*I=0$.
\end{lemma}

\begin{proof}
   If $I*I=0$, then clearly $I_i*I_i=0$ for all $i\in \{1,...,n\}$. The converse follows by \cref{lemchiav}.
\end{proof}

Now we are able to provide the main result of the section. First, recall that if $B$ is a skew left brace, then it is said to be \emph{Artinian} if every descending chain of ideals of $B$ is eventually stationary (see \cite[Section $3$]{jespers2021radical}).

\begin{theor}\label{princi}
    Let $B_1,B_2$ be skew left braces and $B:=B_1\rtimes_{\alpha} B_2$ be the skew left braces semidirect product via $\alpha$, and suppose that $B_1$ is Artinian. If $B$ is semiprime, then so is $B_1$.
\end{theor}

\begin{proof}
    Suppose that $B_1$ is not semiprime, and let $K$ be a nonzero ideal of $B_1$ such that $K*K=0$. Since $B_1$ is Artinian, $K$ contains a minimal ideal, which we denote by $I$. Now, let $J:=\Sigma_{b\in B_2}\alpha_b(I)$. By a standard calculation, one can show that $J$ is an ideal of $B_1$ and is invariant under the action $\alpha$; thus by $3)$ of \cref{projid} $J\times \{0\}$ is an ideal of $B_1\rtimes_{\alpha} B_2$. We check that $J*J=0$: it suffices to check that $j*j'=0$ for all $j,j'\in J$. One has $j=i_1+...+i_n$ and $j'=i_1'+...+i_m'$ for suitable elements $n,m\in \mathbb{N}$, $b_1,...,b_n,b_1',...,b_m'\in B_2$ and $i_s\in \alpha_{b_s}(I),i_t'\in \alpha_{b_t'}(I)$ where $s\in \{1,...,n\},t\in \{1,...,m\}$. Up to enlarging $n$ and to taking some of the $i_s$'s and $i_s'$'s as zero, we may suppose $n=m$ and $b_s=b_s'$. Let $I_s:=\alpha_{b_s}(I)$ for all $s\in \{1,...,n\}$, then $I_s*I_s=0$ follows from \cref{lempre}, and $(I_1+....+I_n)*(I_1+....+I_n)=0$ follows from \cref{lemind2}. Therefore $j*j'=0$. This proves $J*J=0$, whence one also has
    $$(J\times \{0\})*(J\times \{0\})=(J*J)\times \{0\}=\{(0,0)\}$$
    but this contradicts the semiprimality of $B$. 
\end{proof}

\noindent Since finite skew left braces are Artinian, the following is immediate.

\begin{cor}
        Let $B_1,B_2$ be finite skew left braces and $B:=B_1\rtimes_{\alpha} B_2$ the skew left braces semidirect product via $\alpha$. If $B$ is semiprime, then so is $B_1$.
\end{cor}

In general, an analog of \cref{princi} for $B_2$ does not hold even if we only consider the finite case, as the following example shows.

\begin{ex}\label{contro1}
    Let $B_1$ be the skew left brace given by $B_1:=\mathcal{SB}_{12,22}$ and $B_2$ be the trivial skew left brace on the cyclic group of size $2$. By computer calculations, we obtain that $Aut(B_1,+,\circ)\cong (\mathbb{Z}/2\mathbb{Z},+)$ and the unique non-trivial semidirect product $B_1\rtimes B_2$ is semiprime.
\end{ex}

\section{Semidirect product of strongly semiprime skew left braces}

In this small section, we show the “strongly" version of \cref{teokin}. At first, we give a preliminary lemma. 

\begin{lemma}\label{lemid}
    Let $B$ be a skew left brace, $I$ be an ideal of $B$, and $X,Y$ be subgroups of $(B,+)$. Then, we have $(X+I)*(Y+I)\subseteq (X*Y)+I$.
\end{lemma}

\begin{proof}
Since $I$ is an ideal, we have that $(x+i)*(y+j)-x*y\in I$ for all $x\in X$, $y\in Y$, and $i,j\in I$; therefore $(x+i)*(y+j)\in  (X*Y)+I$, for all $x\in X$ and $y\in Y$. Moreover, since $(X*Y)+I$ is a subgroup of $(B,+)$ and since $(X+I)*(Y+I)$ is generated by the elements of the form $(x+i)*(y+j) $, the statement follows.
\end{proof}

\begin{theor}\label{semipri}
    Let $B$ be a skew left brace and $I$ an ideal of $B$. If $I$ and $B/I$ are strongly semiprime skew left braces, then so is $B$.
\end{theor}

\begin{proof}
    Let $J$ be an ideal of $B$. If there is a $*$-product of copies of $J$ that is zero, by repeatedly using \cref{lemid}, we obtain that there exists a $*$-product of copies of the ideal $(J+I)/I$ of $B/I$ that is zero. Since $B/I$ is semiprime, it follows that $(J+I)/I=I/I$ and hence $J\subseteq I$. By semiprimality of $I$ the equality $J=0$ follows. Therefore, $B$ is strongly semiprime.
\end{proof}

\begin{rem}
    \cref{lemid} is a more general version of the idea used in \cite[Section $2$]{kinnear2021wreath} to show \cref{teokin}, where the inclusion $(J+I)*(J+I)\subseteq J*J+I$ for two arbitrary ideals $I,J$ was used. In this context, since we have to consider an arbitrary $*$-product and, in general, the set $X*Y$ is not an ideal even if $X$ and $Y$ are, we have to weaken the hypothesis.
\end{rem}

\noindent As a corollary of the previous theorem, the main result of the section follows.

\begin{cor}
   Let $B_1,B_2$ be strongly semiprime skew left braces and $B:=B_1\rtimes_{\alpha}B_2$ be a semidirect product of $B_1$ and $B_2$ via a homomorphism $\alpha.$ Then, $B$ is a strongly semiprime skew left brace.
\end{cor}

\begin{proof}
   If we set $I:=B_1\times \{0\}$, it follows that $B/I\cong B_2$, hence the statement follows by \cref{semipri}.
\end{proof}

As in the semiprime case, if $B:=B_1\rtimes_{\alpha}B_2$ is a strongly semiprime skew left brace, $B_2$ in full generality is not a strongly semiprime skew left brace. Indeed, \cref{contro1} is a strongly semiprime skew left brace and $B_2*B_2=0$. However, the question remains open for $B_1$.

\begin{question}
    If $B_1$ and $B_2$ are skew left braces such that their semidirect product via $\alpha$ $B_1\rtimes_{\alpha} B_2$ is strongly semiprime, is $B_1$ strongly semiprime?
\end{question}

\section{Strongly prime skew left braces by semidirect products}

In this section, we study semidirect products of simple skew left braces. First, we show the possible ideals of these skew left braces. Later, we apply this result to provide sufficient conditions that allow to construct strongly prime non-simple skew left braces. All these examples answer in negative sense to \cite[Question 1]{smoktunowicz2024more}.

\smallskip 

From now on, a homomorphism $\alpha:G\longrightarrow H$ is said to be \emph{non-trivial} if the kernel of $\alpha$ is a proper subgroup of $G$.

\begin{prop}\label{propo1}
    Let $B_1,B_2$ be simple skew left braces and $\alpha:B_2\longrightarrow Aut(B_1,+,\circ)$ be a non-trivial homomorphism. Suppose that $\bar{J}$ is an ideal of the skew left braces semidirect product $B:=B_1\rtimes_{\alpha} B_2$. Then, $\bar{J}$ is either a trivial ideal, or $\bar{J}=B_1\times \{0\}$ or $\bar{J}$ is such that $\bar{J}\cong B_2$ (as skew left braces) and $B$ is isomorphic, as a  skew left brace, to the skew left braces direct product $(B_1\rtimes \{0\})\times \bar{J}$.
\end{prop}

\begin{proof}
Clearly, $\bar{J}\cap (B_1\rtimes \{0\})$ is an ideal of $B_1\rtimes \{0\}$, therefore, since $B_1$ is simple, this ideal is either $\{(0,0)\}$ or $B_1\rtimes \{0\} $. If it is equal to $B_1\rtimes \{0\} $, by the correspondence theorem we have that $\bar{J}=B_1\rtimes \{0\}$ or $\bar{J}=B_1\rtimes B_2$. Now, suppose that $\bar{J}\cap (B_1\rtimes \{0\})=\{(0,0)\}$. Thus, either $\bar{J}=\{(0,0)\} $ or $\bar{J}$ is a proper ideal of $B$. In the latter case, we obtain that $(B_1\times \{0\})+\bar{J}/B_1\times \{0\}\cong \bar{J}/(B_1\times \{0\})\cap \bar{J}\cong \bar{J}$, and since $B_2\cong B/(B_1\times \{0\})$, we have that $B_2$ contains an ideal isomorphic to $\bar{J}$. Since $\bar{J}\neq \{(0,0)\}$, we obtain that $\bar{J}\cong B_2$. Finally, since $B_2$ is simple and the ideals containing $B_1\times \{0\}$ correspond to the ideals of $B/(B_1\times \{0\})$, the unique ideals of $B$ containing $B_1\times \{0\}$ are $B_1\times \{0\}$ and $B$, hence $\bar{J}+(B_1\rtimes \{0\})=\bar{J}\circ (B_1\rtimes \{0\})=B$, and the statement follows by the equality $\bar{J}\cap (B_1\rtimes \{0\})=\{(0,0)\}$.
\end{proof}

\noindent Now, we want to provide sufficient conditions to ensure that $B_1\rtimes_{\alpha} B_2$ has a unique non-trivial ideal. If $(G,\cdotp)$ is a group and $g\in G$, we denote by $i_g$ the inner automorphism of $(G,\cdotp)$ given by $i_g(h):=g^{-1}\cdotp h \cdotp g$ for all $h\in G$.

\begin{prop}\label{ipo1}
    Let $B_1,B_2$ be simple skew left braces and $\alpha:B_2\longrightarrow Aut(B_1,+,\circ)$ be a non-trivial homomorphism. Suppose that there is a non-trivial ideal $J$ of $B_1\rtimes_{\alpha} B_2$ different from $B_1\times \{0\}$. Then, for every $(a_1,a_2)\in J$ we have $\alpha_{a_2}=i_{a_1}$.
\end{prop}

\begin{proof}
    If $(a_1,a_2)\in J$, by \cref{propo1} we have $(a_1,a_2)\circ (b_1,0)=(b_1,0)\circ (a_1,a_2)$ for all $b_1\in B_1$, which implies the equality $(a_1\circ \alpha_{a_2}(b_1),a_2 )=(b_1\circ a_1, a_2)$ for all $b_1\in B_1$, therefore $\alpha_{a_2}=i_{a_1}$ and the statement follows.
\end{proof}

\noindent Now, we give two corollaries of \cref{ipo1} that are useful for our purposes.

\begin{cor}\label{corol1}
    Let $B_1,B_2$ be simple skew left braces and $\alpha:B_2\longrightarrow Aut(B_1,+,\circ)$ be a non-trivial homomorphism. Moreover, suppose that $\alpha$ is not injective and $(B_1,\circ)$ has trivial center. Then, the unique non-trivial ideal of the skew left braces semidirect product $B_1\rtimes_{\alpha} B_2$ is $B_1\times \{0\}$.
\end{cor}

\begin{proof}
Suppose that $J$ is a non-trivial ideal of $B$ different from $B_1\times \{0\}$. By \cref{projid} and \cref{propo1}, we have that $\pi_2(J)=B_2$, therefore for every $a_2\in B_2$ there exists $a_1\in B_1$ such that $(a_1,a_2)\in J$. Then, by \cref{ipo1}, for every $(a_1,a_2)\in J$ we have $\alpha_{a_2}=i_{a_{1}}$. If we take $a_2\in Ker(\alpha)\setminus \{0\}$, being the center of $(B_1,\circ)$ trivial, it follows that $(0,a_2)\in J$. Thus, we have that $J\cap (\{0\}\times B_2)$ is a non-trivial ideal of $\{0\}\times B_2$, hence $J=\{0\}\times B_2$, thus $\alpha_{a_2}=i_0=id_{B_2}$ for all $a_2\in B_2$; but this contradicts that $\alpha$ is a non-trivial homomorphism.
\end{proof}

\begin{cor}\label{coro2}
    Let $B_1,B_2$ be simple skew left braces and $\alpha:B_2\longrightarrow Aut(B_1,+,\circ)$ be a non-trivial homomorphism. Moreover, suppose that $\alpha_a$ is not an inner automorphism of $(B_1,\circ)$, for every $a\in B_2\setminus \{0\}$. Then, the unique non-trivial ideal of the skew left braces semidirect product $B_1\rtimes_{\alpha} B_2$ is $B_1\times \{0\}$.
\end{cor}

\begin{proof}
If $J$ is a non-trivial ideal of $B_1\rtimes_{\alpha} B_2$ different from $B_1\times \{0\}$, then by \cref{ipo1} every non-zero element of $J$ has the form $(a_1,a_2)\in B_1\times B_2$ with $\alpha_{a_2}=i_{a_1}$. Since $ \alpha_{a_2}$ is not an inner automorphism for $a_2\neq 0$, we must have $a_2=0$ and hence $J\subseteq B_1\times \{0\}$. Therefore, $J= B_1\times \{0\} $ because $B_1$ is simple.
\end{proof}

\noindent We remark that \cref{coro2} uses the main idea of \cite[Proposition 5.1]{CJO20}, where skew left braces with a unique non-trivial ideal were constructed by an outer automorphism of the multiplicative group of $B_1$.\\
Finally, we can construct the desired family of skew left braces. First, we give a simple but fundamental lemma.

\begin{lemma}\label{preparatory}
    Let $B$ be a skew left brace, and suppose that $B$ has a unique non-trivial ideal $J$, that is simple and non-trivial as a skew left brace. Then, $B$ is strongly prime.
\end{lemma}

\begin{proof}
     Every product $I_1*I_2$ of non-zero ideals of $B$ contains $J*J$, which is an ideal of $J$ (see \cite[Section 4]{KoSmVe18}) and hence equal to $J$ by hypothesis. It follows that an arbitrary product of non-zero ideals in $B$ contains $J$. Thus, the statement follows.
\end{proof}

\begin{theor}\label{final}
    Let $B_1,B_2$ and $\alpha$ be as in the statement of \cref{corol1} (resp. \cref{coro2}). Moreover, suppose that $B_1$ is not a trivial skew left brace. Then, the skew left braces semidirect product $B:=B_1\rtimes_{\alpha} B_2$ is strongly prime and non-simple.
\end{theor}

\begin{proof}
By \cref{corol1} (resp. \cref{coro2}), the unique non-trivial ideal is $B_1\rtimes \{0\}\cong B_1$, which is a simple non-trivial skew left brace by hypothesis. Thus, the statement follows by \cref{preparatory}.
\end{proof}

All the examples obtained by \cref{final} are non-simple and strongly prime. Note that if we take $B_1, B_2$ and $\alpha $ as in \cref{coro2} with $B_1$ and $B_2$ of abelian type, $B_1$ non-trivial, and $B_2$ having a prime number of elements, we find the skew left braces of \cite[Proposition 5.1]{CJO20}.

\section{Examples and counterexamples}

In this section, we collect suitable examples and counterexamples involving semiprime skew left braces. 

\smallskip

By definition, we know the following implications: “strongly prime $\Rightarrow$ prime $\Rightarrow$ semiprime" and “strongly semiprime $\Rightarrow$ semiprime". Apart from some special cases (see for example \cite[Section $5$]{trappeniers2023two}), these notions do not coincide, even when we consider skew left braces of abelian type. In \cite[Example 5.7]{KoSmVe18} and \cite{puljic2021some}, examples that show “semiprime $\nRightarrow$ prime" and “prime $\nRightarrow$ strongly prime" were provided. Note that some skew left braces constructed in \cite{puljic2021some} also show “semiprime $\nRightarrow$ strongly semiprime" and “prime $\nRightarrow$ strongly semiprime". For example, let $B:=\mathcal{B}_{81,804}$. Then, $B$ has only $B*B$ as non-trivial ideal, which is not trivial as skew left brace. Therefore 
$B$ is prime. However, by \cite[Proposition $4.4$]{CeSmVe19}, $B$ is not strongly semiprime.

\noindent Now, we use \cref{final} to show that there exist examples of strongly prime non-simple skew left braces. Since these examples are of abelian type, they respond negatively to \cite[Question 1]{smoktunowicz2024more}.

\begin{prop}\label{esem}
    There exists a strongly prime skew left brace of abelian type that is not a simple skew left brace.
\end{prop}

\begin{proof}
By \cite[Section $7.1$]{bachiller2015extensions}, there exists a simple skew left brace $B_1$ of abelian type having $24$ elements. Let $A_2$ and $A_3$ be the Sylow subgroups of $(B_1,+)$. By \cite[Section $3$]{rump2019construction}, we know that $A_2$ and $A_3$ are left ideals of $(B_1,\circ)$. If $A_2$ is normal in $(B_1,\circ)$, by \cite[Lemma $1.8$ and $1.9$]{CeSmVe19} we have that it is an ideal of $B_1$, which is a contradiction. In the same way, one can show that $A_3$ cannot be normal in $(B_2,\circ)$. By a standard exercise, we have that the unique group of order $24$ without normal Sylow subgroups is $\mathbb{S}_4$, therefore $(B_1,\circ)\cong \mathbb{S}_4$. Now, by \cite[Example $4$]{rump2019construction} there exists a unique element $b\in B_1$, which has multiplicative order $2$, such that $i_b(A_2)=A_2$ and $i_b(A_3)=A_3$. By \cite[Examples $3$ and $4$]{rump2019construction}, we obtain the following statements: $b$ does not act trivially on $A_3$ by the map $\lambda_b$; the map $i_b$ induces two skew left braces automorphisms on $A_2$ and $A_3$, regarded as skew left braces. Using these facts, by a standard calculation, we obtain that $i_b$ is an automorphism of $(B_1,+,\circ)$ of order $2$.\\
Now, let $B_2$ be an isomorphic copy of $B_1$. If $\alpha$ is the non-trivial homomorphism from $(B_2,\circ) $ to $Aut(B_1,+,\circ) $ that sends a transposition of $(B_2,\circ)$ to $i_b$ and with $Ker(\alpha)\cong \mathbb{A}_4$, by \cref{corol1} and \cref{final} we obtain that $B_1\rtimes_{\alpha} B_2$ is a non-simple strongly prime skew left brace.
\end{proof}

\noindent We highlight that, by \cref{preparatory}, the example given in \cite[Example 5.3]{CJO20} is also a non-simple strongly prime skew left braces of abelian type, and it was obtained without computer calculations. However, it is a skew left brace with $92.160$ elements, while our example has $576$ elements.\\
Anyway, using the GAP package \cite{Ve24pack}, one can obtain further examples of skew left braces which can be used in \cref{corol1} and \cref{final} to construct further examples of non-simple strongly prime skew left braces of abelian type. For example, we can take $B_1$ and $B_2$ in $\{\mathcal{B}_{24,94},\mathcal{B}_{72,475}\}$.
Another attempt is the inspection of the skew left braces of abelian type having small size by computer \cite{Ve24pack}. In this regard, the following result shows that there are strongly prime skew left braces that are different from the ones considered in this paper.

\begin{prop}
        There exists a strongly prime skew left brace of abelian type that is not simple, and is different from the ones provided up to now. In particular, it is not a semidirect product of skew left braces, and is the minimal example of a strongly prime non-simple skew left brace of abelian type.
\end{prop}

\begin{proof}
Let $B$ be the skew left brace given by $B:=\mathcal{B}_{48,1532}$. Then, $B$ has only one non-trivial ideal $I$ of size $24$, which is isomorphic as a skew left brace to $\mathcal{B}_{24,94}$. By \cref{preparatory}, $B$ is a strongly prime skew left brace. Now, if $B$ is a semidirect product of skew left braces, it must be constructed as a semidirect product of the form $I\rtimes_{\alpha} B_2$, where $B_2$ is the trivial skew left brace of size $2$. By computer calculations \cite{Ve24pack}, we obtain that there is a unique semidirect product of the form $I\rtimes_{\alpha} B_2$ with $Ker(\alpha)\neq B_2$, and it has an ideal of size $2$. Therefore, such a semidirect product cannot be isomorphic to $B$. Again by computer \cite{Ve24pack}, combined with \cite[Proposition $4.4$]{CeSmVe19} to optimize calculations, we find that $B$ is the minimal example of a strongly prime non-simple skew left brace of abelian type.
\end{proof}

\smallskip

The notion of strongly prime (resp. semiprime) skew left brace was originally given for skew left braces of abelian type. This is the reason why, until now, we focused on this case. However, if we consider skew left braces with an arbitrary additive group, we can find a smaller example of strongly prime non-simple skew left brace. Indeed, let $B$ be the skew left brace given in \cref{contro1}. Then, it is a skew left brace of size $24$ with a unique ideal $I$ of size $12$, which is simple and non-trivial as a skew left brace. By \cref{preparatory} $B$ is strongly prime. 

\bibliographystyle{elsart-num-sort}
\bibliography{Bibliography2}

\end{document}